\documentclass[reqno, 11pt]{article}
\usepackage{amssymb, amsmath, amsthm, amsfonts, amscd, epsfig}
\usepackage{color}
\usepackage{subfigure}
\usepackage[english]{babel}
\usepackage{bbm}
\usepackage{graphicx, epsfig}
\usepackage{geometry,graphicx,pgfplots,caption}
\usepackage[export]{adjustbox}
\usepackage[titletoc,toc,title]{appendix}
\usepackage[normalem]{ulem} 

\newtheorem{thm}{Theorem}[section]

\newtheorem{lemma}[thm]{Lemma}

{\bf}{\it}

\newcommand{\Bc}{\mathbf{c}}
\newcommand{\Bd}{\mathbf{d}}

\newcommand{\Om}{\Omega}

\newcommand{\Bx}{\mathbf{x}}

\newcommand{\hBx}{\hat{\mathbf{x}}}
\newcommand{\By}{\mathbf{y}}
\newcommand{\tBy}{\tilde{\mathbf{y}}}
\newcommand{\Bz}{\mathbf{z}}
\newcommand{\tBz}{\tilde{\mathbf{z}}}

\newcommand{\RR}{\mathbb{R}}

\newcommand{\NN}{\mathbb{N}}

\newcommand{\Scal}{\mathcal{S}}

\newcommand{\pa}{\partial}
\newcommand{\eps}{\varepsilon}

\newcommand{\tvarphi}{\tilde{\varphi}}

\newcommand{\omg}{\omega}

\newcommand{\inc}{\textrm{inc}}
\newcommand{\scat}{\textrm{scat}}

\newcommand{\rmi}{\mathrm{i}}
\newcommand{\rme}{\mathrm{e}}

\newcommand{\bke}[1]{\left( #1 \right)}

\newcommand{\p}{\partial}

\newcommand{\ds}{\displaystyle}
\newcommand{\eqnref}[1]{(\ref {#1})}

\newcommand{\beq}{\begin{equation}}
\newcommand{\eeq}{\end{equation}}
\newcommand{\RN}[1]{%
\textup{\uppercase\expandafter{\romannumeral#1}}%
}

\numberwithin{equation}{section}
\numberwithin{figure}{section}

\begin{document}

\title{Monostatic imaging of an extended target with MCMC sampling\thanks{\footnotesize This study was supported by National Research Foundation of Korea (NRF) grant funded by the Korean government (MSIT) (NRF-2021R1A2C1011804)}}

\author{Jiho Hong\thanks{\footnotesize Department of Mathematics, The Chinese University of Hong Kong, Shatin, New Territories, Hong Kong SAR, P.R. China (jihohong@cuhk.edu.hk).}
\and Sangwoo Kang\thanks{\footnotesize Department of Mathematical Sciences, Korea Advanced Institute of Science and Technology, Daejeon 34141, Republic of Korea (sangwoo.kang87@gmail.com).}
\and Mikyoung Lim\thanks{\footnotesize Department of Mathematical Sciences, Korea Advanced Institute of Science and Technology, Daejeon 34141, Republic of Korea (mklim@kaist.ac.kr).}}

\maketitle

\begin{abstract}
We consider the imaging of a planar extended target from far-field data under a monostatic measurement configuration, in which the data is measured by a single moving transducer, as frequently encountered in practical application.  
In this paper, we develop a Bayesian approach to recover the shape of the extended target with MCMC sampling, where a new shape basis selection is proposed based on the shape derivative analysis for the measurement data.
In order to optimize the center and radius of the initial disk, we use the monostatic sampling method for the center and the explicit scattered field expression for disks for the radius. 
Numerical simulations are presented to validate the proposed method. 
\end{abstract}


\section{Introduction}


We consider a scattering problem with a sound-soft obstacle, namely $\Om$, embedded in a homogeneous background medium in two dimensions. We denote by $\eps_0$ and $\mu_0$ the electric permittivity and the magnetic permeability of the background.
We let the incident plane wave be given by a plane wave with an angular frequency $\omega$ so that its wavenumber and wavelength in the background medium are $k=\omg\sqrt{\eps_{0}\mu_{0}}$ and $\lambda=2\pi/k$ respectively. In other words,
\beq\label{u:inc}
u^{\inc}(\mathbf{x};\Bd) = e^{{\rm i}k\mathbf{d}\cdot\mathbf{x}}
\eeq
for some direction vector $\mathbf{d}\in S^1$. 
We assume that $\Om$ is an extended target (that is, the size of $\Om$ is $>\frac{\lambda}{2}$) with a smooth boundary and that $k^2$ is not an interior Dirichlet eigenvalue for $\Om$. 
We denote by $u^{\scat}$ the scattered field due to the obstacle $\Om$. 
Then the total field $u=u^{\inc}+u^{\scat}$ satisfies the Helmholtz equation
\beq\label{HelmholtzEquation}
\left\{
\begin{aligned}
(\Delta + k^2) u &= 0\quad\mbox{in }\mathbb{R}^2\backslash\overline{\Om},\\
u &= 0\quad\mbox{on }\partial\Omega,
\end{aligned}
\right.
\eeq
where the scattered field satisfies the Sommerfeld radiation condition
\begin{equation}\label{SmmerfeldRadiationCondition}
\lim_{|\mathbf{x}|\to\infty} \sqrt{|\mathbf{x}|}\left(\frac{\pa u^{\scat}(\mathbf{x})}{\pa|\mathbf{x}|} - \rmi k u^{\scat}(\mathbf{x})\right) =0
\end{equation}
uniformly in the direction $\hat{\mathbf{x}}={\mathbf{x}}/{|\mathbf{x}|}$ and
\begin{equation}\label{FarFieldPattern}
u^{\scat}(\mathbf{x}) = \frac{\rme^{\rmi k |\Bx|}}{\sqrt{|\mathbf{x}|}}\left( u^{\infty}(\hat{\mathbf{x}},\mathbf{d}) + O\bke{\frac{1}{|\mathbf{x}|}}\right),\quad |\Bx|\rightarrow\infty,
\end{equation}
with the so-called {\it far-field pattern} $u^{\infty}(\hat{\mathbf{x}},\mathbf{d})$ for $(\hat{\mathbf{x}},\mathbf{d})\in\mathbb{S}^1\times\mathbb{S}^1$ (for a fixed $k$).
We may write $u^{\infty}[\Om](\hat{\mathbf{x}},\mathbf{d})$ when it is necessary to indicate the target. 

\begin{figure}[b!]
\centering
\subfigure[\label{MSR}MSR matrix]{
$\ds \hskip 5mm
\begin{bmatrix}
	\ds  * & * &\cdots &*\\[1mm]
	\ds *&* & \cdots &*\\[1mm]
	\vdots&\vdots &\ddots &\vdots \\[1mm]
	*& *& \cdots&\ds *  \end{bmatrix}_{N\times N}$
}  \hskip 1cm
	\subfigure[\label{Mono}Monostatic]{
$\ds \hskip 5mm
\begin{bmatrix}
	\ds  * & & &\\[1mm]
	\ds &* &  &\\[1mm]
	& &\ddots & \\[1mm]
	& & &\ds *  \end{bmatrix}_{N\times N}$
}  
\caption{(a) describes the MSR matrix, where the numbers of incident waves and measurement directions are both $N$; (b) indicates the measurement data in monostatic configuration.}
\label{3matrices}
\end{figure}

\begin{figure}[htp!]
	\centering
	\includegraphics[width=0.45\textwidth]{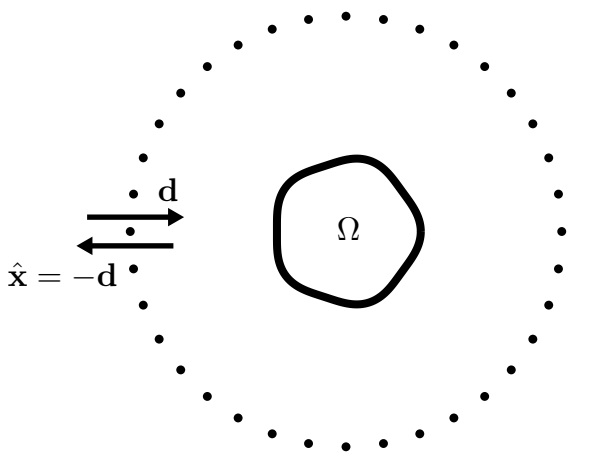}
	\caption{Monostatic measurement configuration. The dots on the big circle describes the directions of a moving transducer which is assumed to be infinitely far from the target $\Om$. The direction $\mathbf{d}$ of the plane wave is always the opposite direction of the direction $\hat{\mathbf{x}}$ of the receiver.}
	\label{Location}
\end{figure}

With applications on non-invasive imaging in various contexts, the inverse scattering problem of acoustic or electromagnetic waves have been studied.
For the inverse scattering problems using full-aperture measurement, various sampling methods have been proposed (see, for example, the survey paper \cite{Potthast:2006:SPM}).
Some methods are based on the shape derivative analysis \cite{Kirsch:1993:DTA, Ammari:2012:MIE}.
Recently, the methods for inverse scattering problems using partial measurement were also developed \cite{Li:2020:ESB,Li:2021:QBA,Yang:2023:BAL,Bektas:2016:DSM,Kang:2019:AID,Kang:2022:MSM}.

Our main focus of this paper is the inverse scattering problem in monostatic configuration; that is, the data is measured by one moving transducer (see Figure \ref{Location}). 
The measurements are 
\beq\label{data:mono}
\left\{u^\infty(\hat{\mathbf{x}},\mathbf{d})\,:\,\hat{\mathbf{x}}=-\mathbf{d},\ \mathbf{d}=\mathbf{d}_j,\ j=1,\dots,M\right\}\quad\mbox{for some }M.
\eeq
In monostatic configuration, the input data is restricted to the diagonal elements comparing to the generic multi-static response (MSR) matrix where all entries in the matrix are given, as illustrated in Figure \ref{3matrices}. As a consequence, it becomes more challenging to successfully recover the target  than in multi-static configuration and, in particular, the singular value decomposition based methods (e.g., the MUSIC algorithm, linear sampling method, factorization method, and subspace migration) have unsatisfactory performance even for small targets. 
In \cite{Bektas:2016:DSM,Kang:2019:AID,Kang:2022:MSM}, direct sampling methods (DSM) for the monostatic configuration are developed so that one can successfully recover small targets. Using the DSM, one obtains the location and comparable magnitude of extended targets.
In this paper, we develop a shape recovery scheme of an extended target from the measurement data given by \eqnref{data:mono}, assuming that the location (a point located near to the center of mass of the target) and comparable size of the target are previously obtained.
We first derive the integral form of the shape derivative of the data pattern and use it to define our basis for shape perturbation.
We propose a new monostatic imaging method based on the Bayesian approach called Markov Chain Monte Carlo (MCMC) sampling.

\section{Layer potential operators}

We denote by $\Gamma^k$ the fundamental solution to the Helmholtz equation $(\Delta+k^2)u=0$ in two dimensions; that is,
$$\Gamma^k(\mathbf{x})=-\frac{\rm i}{4} H_0^{(1)}(k|\mathbf{x}|),\quad\mathbf{x}\in\mathbb{R}^2\backslash\{0\},
$$
where $H_0^{(1)}$ is the Hankel function of the first kind of order $0$.
For a Lipschitz domain $D$ and $\varphi\in L^2(\p D)$, the single- and double-layer potentials are defined by
\begin{align}\label{def:S:helm}
\mathcal{S}_D^k[\varphi](\mathbf{x})&=\int_{\p D}{\Gamma^k(\mathbf{x}-\mathbf{y})}\varphi(\mathbf{y})\,d\sigma(\mathbf{y}),\quad \mathbf{x}\in\RR^2\backslash\p D,
\\
\label{def:D}
\mathcal{D}_D^k[\varphi](\mathbf{x})&=\int_{\p D}\frac{\p\Gamma^k(\mathbf{x}-\mathbf{y})}{\p\nu_\mathbf{y}}\varphi(\mathbf{y})\,d\sigma(\mathbf{y}),\quad \mathbf{x}\in\RR^2\backslash\p D,
\end{align}
where $\nu_\mathbf{y}$ is the unit outward normal vector to $\p D$ at $\mathbf{y}$.
The functions \eqnref{def:S:helm} and \eqnref{def:D} satisfy the Helmholtz equation in $\mathbb{R}^2\backslash\p D$ and  admit the Sommerfeld radiation condition at infinity.
We also define
\begin{align}
\label{def:K}
\ds\mathcal{K}_D^k[\varphi](\mathbf{x})&= \int_{\p D}\frac{\p\Gamma^k(\mathbf{x}-\mathbf{y})}{\p\nu_{\mathbf{y}}}\,\varphi(\mathbf{y})\,d\sigma(\mathbf{y}),\quad \mathbf{x}\in\p D,\\
\label{def:Kstar}\mathcal{K}_D^{k,*}[\varphi](\Bx)&=\int_{\p D}\frac{\p\Gamma^k(\mathbf{x}-\mathbf{y})}{\p\nu_\mathbf{x}}\,\varphi(\mathbf{y})\,d\sigma(\mathbf{y}),\quad \mathbf{x}\in\p D.
\end{align}
We refer to \cite{Colton:2013:IEM} for the properties of the layer potentials for the Helmholtz equation. 

Since $k^2$ ($k>0$) is not a Dirichlet eigenvalue of $\Om$, the boundary integral operator $\mathcal{S}_\Omega^k$ is invertible as an operator from $L^{2}(\p\Om)$ to $H^{1}(\p\Om)$ (see, for instance, \cite[Proposition 7.9, Chapter 9]{Taylor:1996:PDE} or \cite{Ammari:2004:SRS}). 
One can express the solution to \eqnref{HelmholtzEquation}--\eqnref{SmmerfeldRadiationCondition} with $u^{\inc}(\Bx;\Bd)$ given by \eqnref{u:inc} with a direction vector $\mathbf{d}\in S^1$ as 
$$u=u^{\inc}+\mathcal{S}_\Omega^k[-\varphi]\quad \mbox{in }\mathbb{R}^2\backslash\overline{\Om}$$
with 
\beq\label{varphi:def}
\varphi(\cdot;\Bd)= (\mathcal{S}_\Omega^k)^{-1} \left[ u^{\inc}\big|_{\p\Om}\right]\quad\mbox{on }\p\Om.
\eeq
It then holds that
\beq\label{varphi:int}
\mathcal{S}_\Omega^k\left[\varphi(\cdot;\Bd)\right](\Bx)=u^{\inc}(\Bx;\Bd)\quad\mbox{for }\Bx\in\p\Om.
\eeq
Note that for $\tBy\in\p\Om$, $$|\mathbf{x}-\tBy|=|\mathbf{x}|-\frac{\mathbf{x}\cdot\tBy}{|\mathbf{x}|} + O\left(|\mathbf{x}|^{-1}\right)\quad\mbox{as}\quad |\mathbf{x}|\to\infty.$$
The Hankel function satisfies the asymptotic relation (see, for instance, \cite{Colton:2013:IEM}):
$$H_0^{(1)}(t)=\sqrt{\frac{2}{\pi t}}\,  e^{{\rm i}(t-\pi/4)}\left(1+O\left(\frac{1}{t}\right)\right)\quad\mbox{as }t\to\infty.$$
It follows that 
\begin{align*}
u^{\scat}(\mathbf{x})
&= \int_{\p\Omega}{\frac{\rm i}{4} H_0^{(1)}(k|\mathbf{x}-\tBy|)}\varphi(\tBy;\Bd)\,d\sigma(\tBy)\\
&=\frac{e^{{\rm i}\frac{\pi}{4}}}{\sqrt{8\pi k}}\frac{e^{{\rm i}k|\mathbf{x}|}}{\sqrt{|\mathbf{x}|}}\int_{\p\Omega} e^{-{\rm i}k\hat{\mathbf{x}}\cdot\tBy}\left(1+O(|\mathbf{x}|^{-1})\right)\varphi(\tBy;\mathbf{d})\,d\sigma(\tBy)\quad\mbox{as }|\mathbf{x}|\to\infty
\end{align*}
and, thus,
\beq\label{eq:uinf:shapeder}
u^\infty(\hat{\mathbf{x}},\mathbf{d}) =\frac{e^{{\rm i}\frac{\pi}{4}}}{\sqrt{8\pi k}} \int_{\p\Om} e^{-{\rm i}k\hBx\cdot \tBy}\varphi(\tBy;\mathbf{d})\,d\sigma(\tBy).
\eeq

\section{Shape derivative analysis for the far-field pattern}
Let $\Om$ be a perturbation of a smooth reference domain $\Om_0$, that is,
\beq\label{Om:deform}
\p \Om = \left\{ \By+  h(\By) \nu_0(\By) : \, \By\in \p \Om_0 \right\}
\eeq
with a real-valued $C^1$ function $h$ on $\p\Om_0$, where $\nu_0$ is outward unit normal to $\p \Om_0$. 
We derive the shape derivative for the far-field pattern (refer to \cite{Ammari:2012:GPT,Ammari:2013:MSM} for the shape derivative analysis of the generalized polarization tensors). For notational simplicity, for directions vectors $\Bd\in S^1$, we define density functions $\p\Om_0$ as
\beq\label{psi:d}
\psi_\Bd=(\mathcal{S}_{\Om_0}^k)^{-1}[v|_{\p\Om_0}]\quad\mbox{with }v(\By)=e^{ik{\Bd}\cdot\By}.
\eeq
It is worth noting that, in Theorem \ref{thm:shapeder:farfielddata}, $\Om_0$ is an arbitrary smooth domain and $\hBx,\Bd$ are arbitrary direction vectors, while the theorem is applied assuming that $\Om_0$ is a disk and $\hBx=-\Bd$ in following Sections \ref{sec:gibbs} and \ref{sec:numerical}.

\begin{thm}\label{thm:shapeder:farfielddata}
Fix $k$ and direction vectors $\hat{\Bx},\Bd\in S^1$. Let $\Om$ be given by \eqnref{Om:deform} with $h=\eps h_0$, where $h_0$ is a reference shape deformation function and $\eps$ is a small parameter. For the incident field as $u^{\inc}(\mathbf{x}) = e^{{\rm i}k\mathbf{d}\cdot\mathbf{x}}$, the far-field patterns corresponding to $\Om$ and $\Om_0$ satisfy that
\beq\label{eq:shapeder}
\left(u^\infty[\Om]-u^\infty[\Om_0]\right)(\hat{\Bx},\Bd)=-\frac{e^{{\rm i}\frac{\pi}{4}}}{\sqrt{8\pi k}}  \int_{\p\Om_0} h \psi_{-\hat{\Bx}}\,\psi_{\Bd} \,d\sigma +o\big(\eps\big)\quad\mbox{as }\eps\rightarrow 0.
\eeq
\end{thm}

\begin{proof}

Let $Y(t)$, $t\in[a,b]$ for some $a,b\in\RR$, be the positive oriented parametrization by arc-length for $\p\Om_0$. Then $Y'(t)=T(\By)$ is the tangential vector at $\By\in\p\Om_0$. The outward unit normal to $\p\Om_0$, namely $\nu_0$, is given by $\nu_0(\By)=R_{-\frac{\pi}{2}} Y'(t)$, where $R_{-\frac{\pi}{2}}$ is the rotation by $-\pi/2$. We denote by $\tau(\By)$ the curvature at $\By$ so that
$$Y''(t)=\tau(\By)\,\nu_0(\By).$$
For simplicity, we will sometimes write $h_0(t)$ for $h_0(Y(t))$.
Then, $\tilde{Y}(t)=Y(t)+\eps h_0(t)\nu_0(\By)$ is a parametrization of $\p\Om$. The length element $d\tilde{\sigma}(\tBy)$ of $\p\Om$  admits that (see, for instance, \cite[Section 2]{Ammari:2010:CIP})
\beq\label{dsigma:tildey}
d{\tilde{\sigma}}(\tBy)=\big(1-\eps\tau(\By)\, h_0(\By)+o(\eps)\big)d\sigma(\By),
\eeq
where $\tBy=\By+\eps h_0(\By)\nu_0(\By)$ and $o(\eps)/\eps\rightarrow 0$ as $\eps\rightarrow0$ uniformly in $\By$.
For later use, we define the operator 
\beq\label{T:oper:def}
T:L^2(\p\Om)\rightarrow L^2(\p\Om_0)\quad\mbox{by }(Tg)(\By)=g(\tBy).
\eeq

Set the density functions
\begin{align}\notag
\varphi(\cdot;\Bd)&= (\mathcal{S}_\Omega^k)^{-1} \left[ u^{\inc}\big|_{\p\Om}\right]\quad\mbox{on }\p\Om,\\\label{varphi0:def}
\varphi_0(\cdot;\Bd)&= (\mathcal{S}_{\Omega_0}^k)^{-1} \left[ u^{\inc}\big|_{\p\Om_0}\right]=\psi_\Bd\quad\mbox{on }\p\Om_0,\\\notag
\tvarphi(\cdot;\Bd)&=T\varphi(\cdot;\Bd)\in L^2(\p\Om_0).
\end{align}
Then, by \eqnref{varphi:def} and \eqnref{eq:uinf:shapeder}, it holds that
\beq\label{u:inf:u:0:diff}
\begin{aligned}
&\left(u^\infty[\Om]-u^\infty[\Om_0]\right)(\hat{\mathbf{x}},\mathbf{d}) \\
=&\frac{e^{{\rm i}\frac{\pi}{4}}}{\sqrt{8\pi k}} \left( \int_{\p\Om} e^{-{\rm i}k{\hBx}\cdot \tBy}\varphi(\tBy;\mathbf{d})\,d\tilde{\sigma}(\tBy)
-\int_{\p\Om_0} e^{-{\rm i}k{\hBx}\cdot \mathbf{y}}\varphi_0(\mathbf{y};\mathbf{d})\,d\sigma(\mathbf{y})\right).
\end{aligned}
\eeq
Using \eqnref{dsigma:tildey} and \eqnref{u:inf:u:0:diff}, we derive 
\begin{align}\notag
&\int_{\p\Om} e^{-{\rm i}k{\hBx}\cdot \tBy}\varphi(\tBy;\mathbf{d})\,d\tilde{\sigma}(\tBy)
-\int_{\p\Om_0} e^{-{\rm i}k{\hBx}\cdot \mathbf{y}}\varphi_0(\mathbf{y};\mathbf{d})\,d\sigma(\mathbf{y})\\\notag
=&
\int_{\p\Om_0}\left(e^{-{\rm i}k{\hBx}\cdot \tBy}\varphi(\tBy;\mathbf{d}) \big(1-\eps\tau(\By)\, h_0(\By)\big)
-e^{-{\rm i}k{\hBx}\cdot \mathbf{y}}\varphi_0(\mathbf{y};\mathbf{d})\right)d\sigma(\mathbf{y})+o(\eps)\\\notag
=&\int_{\p\Om_0} \left(e^{-{\rm i}k{\hBx}\cdot \tBy} -e^{-{\rm i}k{\hBx}\cdot \By}\right) \varphi_0(\By;\Bd)\,d\sigma(\By)
+\int_{\p\Om_0} e^{-{\rm i}k{\hBx}\cdot \tBy}\left(\varphi(\tBy;\Bd)-\varphi_0(\By;\Bd)\right)d\sigma(\By)\\\notag
&+\int_{\p\Om_0}(-\eps\tau(\By)\,h_0(\By)) e^{-{\rm i}k{\hBx}\cdot \tBy} \varphi(\tBy;\Bd)\,d\sigma(\By)
+o(\eps)\\ \label{u_uzero_main}
=:&I_1+I_2+I_3+o(\eps).
\end{align}

We first estimate $I_1$ and $I_2$. Note that
\begin{align}\label{plane:differ}
e^{-{\rm i}k{\hBx}\cdot \tBy} - e^{-{\rm i}k{\hBx}\cdot \mathbf{y}}
=&e^{-{\rm i}k{\hBx}\cdot \mathbf{y}}\left(e^{-{\rm i}k{\hBx}\cdot \eps h_0(\By)\nu_0(\By)}  -1\right)
= -{\rm i}k\left(\hBx\cdot\nu_0(\mathbf{y})\right)\eps h_0(\mathbf{y})e^{-{\rm i}k{\hBx}\cdot\mathbf{y}}+o(\eps)
\end{align}
uniformly for $\By\in\p\Om_0$. 
Using \eqnref{plane:differ}, one can easily find that 
\begin{align}
\label{eq:thm:I1}
I_1 & = -\eps\int_{\p\Om_0}  \left[{\rm i}k(\hBx\cdot\nu_0(\mathbf{y}))\,h_0(\mathbf{y})e^{-{\rm i}k{\hBx}\cdot\mathbf{y}}\right]\varphi_0(\mathbf{y};\Bd)\,d\sigma(\mathbf{y}) + o(\eps),\\
\label{eq:thm:I3}
I_3&= -\eps\int_{\p\Om_0}\tau(\mathbf{y}) h_0(\mathbf{y}) e^{-{\rm i}k{\hBx}\cdot \mathbf{y}}\varphi(\tBy;\Bd)\,d\sigma(\mathbf{y})+ o(\eps).
\end{align}

To estimate $I_2$, we use the following decomposition:
\begin{align}\notag
\varphi(\tBy;\Bd)-\varphi_0(\By;\Bd)
=& (T\varphi -\varphi_0)(\By;\Bd)\\\label{second:in:I2}
=&\left[ (\mathcal{S}_{\Omega_0}^k)^{-1}[f](\By)  -\varphi_0(\By;\Bd)\right]+ \left[T\varphi(\By;\Bd) -(\mathcal{S}_{\Omega_0}^k)^{-1} [f](\By)      \right]
\end{align}
with
$$f\in L^2(\p\Om_0)\quad\mbox{given by }f(\By)=e^{{\rm{i}}k\Bd\cdot\tBy}.$$ 
Note that 
\beq\label{SinvTf}
T^{-1}f=u^\inc\big|_{\p\Om}\quad\mbox{and}\quad (\mathcal{S}_{\Omega}^k)^{-1}T^{-1}[f]=\varphi(\cdot;\Bd).
\eeq

In view of \eqnref{plane:differ},  \eqnref{varphi:int} for $\Om_0$, and the jump relation for the single-layer potential from the interior of $\p\Om_0$, it holds that for $\By\in\p\Om_0$,
\begin{align*}
f(\By) - e^{{\rm i}k\Bd\cdot\By} 
= e^{{\rm i}k\Bd\cdot\tBy} - e^{{\rm i}k\Bd\cdot\By} 
&=-{\rm i}k\left(\Bd\cdot\nu_0(\mathbf{y})\right)\eps h_0(\mathbf{y})e^{-{\rm i}k{\Bd}\cdot\mathbf{y}}+o(\eps)\\
&=\eps h_0(\By)\,\frac{\p}{\p\nu_0}\mathcal{S}_{\Om_0}^k[\varphi_0]\Big|_{\p\Om_0}^-(\By) + o(\eps)\\
&=\eps\Big(-\frac{1}{2}h_0\varphi_0+h_0\,\mathcal{K}_{\Om_0}^{k,*}[\varphi_0]\Big)(\By)+ o(\eps)
\end{align*}
and, hence,
\begin{align}\label{I2_1:part}
(\mathcal{S}_{\Omega_0}^k)^{-1}[f](\By)  -\varphi_0(\By;\Bd)
=\eps (\mathcal{S}_{\Omega_0}^k)^{-1}\Big[-\frac{1}{2}h_0\varphi_0+h_0\mathcal{K}_{\Om_0}^{k,*}[\varphi_0(\cdot;\Bd)]\Big](\By)+ o(\eps).
\end{align}
We now consider the second term in \eqnref{second:in:I2}.
From \eqnref{T:oper:def} and \eqnref{SinvTf}, it holds that
\begin{align}\notag
T\varphi -(\mathcal{S}_{\Omega_0}^k)^{-1} [f]
&=\left(T (\mathcal{S}_{\Omega}^k)^{-1}T^{-1} -(\mathcal{S}_{\Omega_0}^k)^{-1}\right)[f]\\\notag
&=(\mathcal{S}_{\Omega_0}^k)^{-1}\left(\Scal_{\Om_0}^k-T\Scal_\Om^k T^{-1}\right)T (\mathcal{S}_{\Omega}^k)^{-1}T^{-1}[f]\\\label{I2_2:part1}
&=(\mathcal{S}_{\Omega_0}^k)^{-1}\left(\Scal_{\Om_0}^k-T\Scal_\Om^k T^{-1}\right)T\varphi\qquad\mbox{on }\p\Om_0.
\end{align}
For $\By\in\p\Om_0$, we have
\begin{align}\notag
&\left(\Scal_{\Om_0}^k-T\Scal_\Om^k T^{-1}\right)[T\varphi](\By)
= \Scal_{\Om_0}^k[\tvarphi](\By)   -\Scal_\Om^k[\varphi](\tBy)\\\notag
=&\int_{\p \Om_0}\Gamma^k (\By-\Bz)\tvarphi(\Bz)\,d\sigma(\Bz)
-\int_{\p\Om}\Gamma^k(\tBy-\tBz)\varphi(\tBz)\,d\tilde{\sigma}(\tBz)\\\notag
=&\int_{\p \Om_0}\Gamma^k (\By-\Bz)\tvarphi(\Bz)\,d\sigma(\Bz)
-\int_{\p\Om_0}\Gamma^k(\tBy-\tBz) \tvarphi(\Bz)\big(1-\eps \tau(\Bz) h_0(\Bz)+o(\eps)\big)d\sigma(\Bz)\\\label{I2_2:part2}
=&\int_{\p\Om_0}\left(\Gamma^k (\By-\Bz)-\Gamma^k(\tBy-\tBz)\right)\tvarphi(\Bz)\,d\sigma(\Bz)
+\int_{\p\Om_0}\Gamma^k(\tBy-\tBz) (\eps\tau \,h_0 +o(\eps))(\Bz)\tvarphi(\Bz)\,d\sigma(\Bz).
\end{align}
Note that	\begin{align*}
\Gamma^k(\mathbf{y}-\mathbf{z}) -\Gamma^k(\tBy-\tBz) & =\frac{\rm i}{4}\left(H_0^{(1)}(k|\tBy-\tBz|)-H_0^{(1)}(k|\mathbf{y}-\mathbf{z}|)\right)\\
&=\frac{\rm i}{4}\,k\left(H_0^{(1)}\right)'(k|\mathbf{y}-\mathbf{z}|)\,\frac{\langle \mathbf{y}-\mathbf{z},\, h(\mathbf{y})\nu(\mathbf{y})- h(\mathbf{z})\nu(\mathbf{z})\rangle}{|\mathbf{y}-\mathbf{z}|}+o(\eps).
\end{align*}
From \eqnref{I2_2:part1} and \eqnref{I2_2:part2}, we obtain
\begin{align}
&T\varphi(\By;\Bd) -(\mathcal{S}_{\Omega_0}^k)^{-1} [f](\By)  \notag\\\notag
=&\eps\,(\mathcal{S}_{\Omega_0}^k)^{-1}\left[-h_0 \mathcal{K}_{\Om_0}^{k,*}\left[\tvarphi\right] -\mathcal{K}_{\Om_0}^k\left[h_0\tvarphi\right] \right](\By)
+(\eps\tau h_0 +o(\eps))(\By)\,\tvarphi(\By)+o(\eps)\\\label{I2_2:part}
=&\eps\,(\mathcal{S}_{\Omega_0}^k)^{-1}\left[-h_0 \mathcal{K}_{\Om_0}^{k,*}\left[T\varphi\right] -\mathcal{K}_{\Om_0}^k\left[h_0 T\varphi\right] \right](\By)
+(\eps\tau h_0 +o(\eps))(\By)\,(T\varphi)(\By)+o(\eps),
\end{align}
where $o(\eps)$ is independent of $T\varphi$.
In \eqnref{I2_2:part}, one can find that $\|T\varphi\|_{H^{1/2}(\p\Om_0)}$ is bounded for sufficiently small $\eps$ (because $\|(\Scal_\Om^k)^{-1}\|_{H^{1/2}(\p\Om)\to H^{-1/2}(\p\Om)}$ is bounded for sufficiently small $\eps$; see, for example, \cite[Theorem 6.1]{Ammari:2004:SRS}).
It then follows from \eqnref{I2_1:part} and \eqnref{I2_2:part} that 
$$\|T\varphi(\cdot;\Bd)-\varphi_0(\cdot;\Bd)\|_{L^2(\p\Om_0)}=O(\eps).$$ 
Hence, in \eqnref{eq:thm:I3} and \eqnref{I2_2:part}, we can replace $T\varphi$ by $\varphi_0$. 
Then, by using \eqnref{I2_1:part} and \eqnref{I2_2:part} again, we derive
\begin{align*}
I_2 & = \int_{\p\Om_0}e^{-{\rm i}k{\hBx}\cdot \mathbf{y}}\left(\varphi(\tBy;\Bd)-\varphi_0(\By;\Bd)\right)d\sigma(\By)\\
& =\eps \int_{\p\Om_0}e^{-{\rm i}k{\hBx}\cdot \mathbf{y}} \, (\mathcal{S}_{\Om_0}^k)^{-1}\left[-\frac{1}{2}h_0\varphi_0-\mathcal{K}_{\Om_0}^{k}[h_0\varphi_0]\right](\mathbf{y})\,d\sigma(\mathbf{y}) - I_3  +o(\eps).
\end{align*}
By \eqnref{eq:thm:I1} and \eqnref{eq:thm:I3}, we arrive at
\begin{align}\notag
&I_1+I_2+I_3 \\\notag
= &\eps\int_{\p\Om_0} e^{-{\rm i}k{\hBx}\cdot \mathbf{y}} \left(-{\rm i}k(\hBx\cdot\nu_0(\mathbf{y}))\,(h_0\varphi_0)(\mathbf{y}) - (\mathcal{S}_{\Om_0}^k)^{-1}\Big[\frac{1}{2}h_0\varphi_0+\mathcal{K}_{\Om_0}^{k}[h_0\varphi_0]\Big](\mathbf{y})\right)\,d\sigma(\mathbf{y}) + o(\eps).\end{align}

We can express the above integral by using the solution to the boundary value problem:
\begin{align*}
\begin{cases}
\ds (\Delta + k^2)v = 0\quad& \mbox{in }\Om_0,\\
\ds v = h_0\varphi_0\quad&\mbox{on }\p\Om_0.
\end{cases}
\end{align*}
Recall the Green's identity:
$v = \mathcal{D}_{\Om_0}^k[v|_{\p\Om_0}] - \mathcal{S}_{\Om_0}^k[\frac{\p v}{\p\nu_0}]$ in $\Om_0.$
The jump relation for the double-layer potential leads to that
$$h_0\varphi_0 = \Big(\frac{1}{2}I + \mathcal{K}_{\Om_0}^k\Big)[h_0\varphi_0] - \mathcal{S}_{\Om_0}^k\left[\frac{\p v}{\p\nu_0}\right]\quad\mbox{on }\p\Om_0.$$
From this relation and the fact that $h\varphi_0=v$ on $\p\Om_0$, we obtain
\begin{align*}
&I_1+I_2+I_3\\
=&\eps\int_{\p\Om_0}  \left[\frac{\p (e^{-{\rm i}k{\hBx}\cdot \mathbf{y}})}{\p\nu_0(\mathbf{y})}\, v(\mathbf{y}) 
- e^{-{\rm i}k{\hBx}\cdot \mathbf{y}}\Big((\mathcal{S}_{\Om_0}^k)^{-1} [h_0\varphi_0] +\frac{\p v}{\p\nu_0}\Big)(\By) \right]\,d\sigma(\mathbf{y})+o(\eps)\\
=& \eps\int_{\Om_0}  \left[\Delta(e^{-{\rm i}k{\hBx}\cdot \mathbf{y}})v(\mathbf{y}) - e^{-{\rm i}k{\hBx}\cdot \mathbf{y}}\Delta v(\mathbf{y})\right]\,d\sigma(\mathbf{y})
- \eps\int_{\p\Om_0} e^{-{\rm i}k{\hBx}\cdot \mathbf{y}}(\mathcal{S}_{\Om_0}^k)^{-1} [h_0\varphi_0] (\mathbf{y})\,d\sigma(\mathbf{y})+o(\eps)\\
=&- \eps\int_{\p\Om_0} e^{-{\rm i}k{\hBx}\cdot \mathbf{y}}(\mathcal{S}_{\Om_0}^k)^{-1} [h_0\varphi_0] (\mathbf{y})\,d\sigma(\mathbf{y})+o(\eps)\\
=&- \eps\int_{\p\Om_0} \mathcal{S}_{\Om_0}^k[\psi_{-\hBx}]\,(\mathcal{S}_{\Om_0}^k)^{-1} [h_0\varphi_0] \,d\sigma+o(\eps).
\end{align*}
In the last equality, we use the symmetricity for the single-layer potential, that is, 
$$\int_{\p\Om_0}\mathcal{S}_{\Om_0}^k[g_1]\, g_2\,d\sigma=\int_{\p\Om_0}g_1 \,\mathcal{S}_{\Om_0}^k[g_2]\,d\sigma \quad\mbox{for }g_1,g_2\in L^2(\p\Om_0).$$
By \eqnref{varphi0:def} and \eqnref{u_uzero_main}, we complete the proof.
\end{proof}

\section{MCMC sampling scheme with a new shape basis}\label{sec:gibbs}

Let $\Om$ be a sound-soft, extended target for a fixed frequency $k$. 
We assume that the far-field pattern of $\Om$ is obtained in monostatic configuration, that is, multiple impinging waves with various $\Bd$ in full-aperture are used and the resulting far-field pattern $u^\infty(\hBx,\Bd)$ is obtained only for $\hBx=-\Bd$.  In other words, the following measurement data is given: for some $J\in\NN$,
$$u^\infty(\hBx_j,-\hBx_j)\quad\mbox{with } \hBx_j=(\cos\theta_j,\sin\theta_j),\ \theta_j=\frac{2\pi j}{J},\ j=1,\dots,J.$$

\subsection{Shape basis functions associated with the measurement configuration}

Reconstruction schemes that use the full entries of the multistatic response matrix rarely shows good imaging  performance in monostatic configuration in which only diagonal entries of the multistatic response matrix are given. Recently, monostatic sampling methods are developed based on small volume expansions for the scattered fields (see, for example, \cite{Kang:2022:MSM}); for the extended target $\Om$, the center $\mathbf{c}_0$ (a point located near the center of mass of the target) and radius $r_0$ (size comparable with the target) can be recovered by the monostatic sampling methods.

From now on, we consider the extended target $\Om$ as a perturbation of 
$$\Om_0=B(\mathbf{c}_0,r_0)$$
given by 
\beq\label{Om:deform2}
\p \Om = \Big\{ \mathbf{c}_0+(\By-\Bc_0)\exp\big(\frac{h(\By)}{r_0}\big)\,:\,\By\in\p \Om_0\Big\}
\eeq
and propose a Bayesian approach to recover the shape details of $\Om$ by selecting a new shape basis based on Theorem \ref{thm:shapeder:farfielddata}. 
We note that $\Om$ given by \eqnref{Om:deform2} is a star-shaped domain and that
\beq\label{exp:nu:formulation}
\mathbf{c}_0 + (\mathbf{y}-\mathbf{c}_0)\exp\Big(\frac{h(\By)}{r_0}\Big) =\mathbf{y}+h(\mathbf{y})\nu_0(\mathbf{y})+o(\|h\|_{C^1(\p B(\Bc_0,r_0))}) \quad\mbox{if }\|h\|_{C^1(\p \Om_0)} \ll 1.
\eeq
The formulation \eqnref{Om:deform2} has the merit to produce a simple curve for any $h$, including the functions that appear in the iterations in Susbsection \ref{subsec:Gibbs}.

\smallskip

When the deformation is small, in view of Theorem \ref{thm:shapeder:farfielddata} and \eqnref{exp:nu:formulation}, the main factors in $h$ that affect $u^\infty(\hBx_j,-\hBx_j)$ are the real and imaginary parts of $\psi_{-\hBx_j}^2$ (see \eqnref{psi:d}).
From this understanding,
we consider the density functions $\Psi_{j}$ in $L^2(\p\Om_0)$ defined by
\beq\label{def:basis:realimag}
\begin{cases}
\ds	\Psi_{2j-1}=\Re \big\{\psi_{-\hBx_j}^2\big\},\\[1mm]
\ds	\Psi_{2j}=\Im \big\{\psi_{-\hBx_j}^2\big\}\quad\mbox{for }j=1,\dots,J,
\end{cases}
\eeq
where $\Re\{\cdot\}$ and $\Im\{\cdot\}$ are the real part and the imaginary part of a complex number, respectively.
We form a shape basis $\{\widetilde{\Psi}_{j}\}_{j=1}^{\widetilde{J}}$ by using the Gram--Schmidt process on $\{\Psi_{j}\}_{j=1}^{2J}$.
In Subsection \ref{subsec:Gibbs}, we propose a Bayesian approach to find a shape deformation function 
\beq\label{htilde:expression}
\widetilde{h}(\mathbf{y})=\sum_{j=1}^{\widetilde{J}}c_j \widetilde{\Psi}_j (\By),\quad\By\in\p B(\Bc_0,r_0),
\eeq
such that the resulting domain, namely $\widetilde{\Om}$, defined by \eqnref{Om:deform2} with $h$ replace by $\widetilde{h}$ has far-field patterns similar to the measurement data. 
Numerical results in Section \ref{sec:numerical} shows that the proposed shape basis provides better imaging performance than the Fourier basis. 


\subsection{MCMC sampling scheme}\label{subsec:Gibbs}
We make prediction on $h$ in \eqnref{Om:deform2} by the following sequence of functions to be defined recursively:
$$h^{(m)}(\mathbf{y})=\sum_{j=1}^{\widetilde{J}} c_j^{(m)} \widetilde{\Psi}_{j}(\mathbf{y})\quad\mbox{for }\mathbf{y}\in\p\Om_0,\ m=0,1,2,\dots,M.$$
The corresponding domain for each $m$ is expressed by
\beq\label{eq:iteration:MCMC:Om}
\p\Om^{(m)} = \Big\{\mathbf{c}_0 + (\mathbf{y}-\mathbf{c}_0)\exp\Big(\frac{h^{(m)}(\mathbf{y})}{r_0}\Big)\,:\,\mathbf{y}\in\p\Om_0\Big\}.
\eeq

We follow a well-known MCMC algorithm called systematic scan Hastings sampler \cite[Algorithm 2.1]{Levine:2005:NMC}.
More precisely, we sample
$$\{c_j^{(m)}\}_{j=1}^{\widetilde{J}},\qquad m=1,2,3,\dots$$
by the following procedure (1.a)-(2.c):
\begin{itemize}
\item[(1.a)]
We choose constant parameters $\beta\in(0,1)$, $\lambda>0$ and $\sigma>0$. 
We also choose a divisor $L$ of $\widetilde{J}$, where $\widetilde{J}=nL$ for some $n\in\mathbb{N}$.
\item[(1.b)]
We set $c_j^{(0)}=0$ for all $j=1,\dots,\widetilde{J}$ (so that $\Om^{(0)}=\Om_0$) and let $m=1$.
\item[(2.a)] We independently choose $x_l^{(m)}\sim N(0,1)$ for $l=1,2,\cdots,L$.
We set
$$\widetilde{c}_{j}^{(m)}=\sqrt{1-\beta}\, c_{j}^{(m-1)} + \sqrt{\beta}\,x_l^{(m)}$$
if $j\equiv l+(m-1)L\mod \widetilde{J}$ for some $l=1,2,\cdots,L$.
Otherwise, we set $\widetilde{c}_{j}^{(m)}=c_{j}^{(m-1)}$.
The resulting MCMC iteration is called a single component sampling when $L=1$, whereas it is called a group sampling when $L\ge2$.

\item[(2.b)]
We define the boundary of the $m$-th auxiliary domain, $\p\widetilde{\Om}^{(m)}$, by the right-hand side of \eqnref{eq:iteration:MCMC:Om} with $h^{(m)}$ replaced by $\sum_{j=1}^{\widetilde{J}} \widetilde{c}_j^{(m)} \Psi_{j}$.
We compute $u^{\infty}[\widetilde{\Om}^{(m)}](\hBx_j,-\hBx_j)$ using our forward solver, and then, compute the energy function
\beq\label{def:prob:foracceptance}
\pi^{(m)}=\exp\bigg(-\frac{1}{2\sigma}\sum_{j=1}^{\widetilde{J}}\left|u^{\infty}(\hBx_j,-\hBx_j) - u^{\infty}[\widetilde{\Om}^{(m)}](\hBx_j,-\hBx_j)\right|^2 - \tau R^{(m)}\bigg)\eeq
where $u^{\infty}(\hBx_j,-\hBx_j)$ is the measurement data and $R^{(m)}$ is the regularization term $$R^{(m)}:=\mbox{length}\left(\p\widetilde{\Om}^{(m)}\right) \int_{\p\widetilde{\Om}^{(m)}}\kappa(\mathbf{y};\p\widetilde{\Om}^{(m)})^2\,d\sigma(\mathbf{y})$$
Here, $\kappa(\mathbf{y};\p\widetilde{\Om}^{(m)})$ is the curvature of $\p\widetilde{\Om}^{(m)}$ at each $\mathbf{y}\in\p\widetilde{\Om}^{(m)}$.
(The term $R^{(m)}$ is invariant under magnification and translation of $\widetilde{\Om}^{(m)}$.)
We then define the acceptance rate $\alpha^{(m)}$ by
$$\alpha^{(m)}=\frac{\pi^{(m)}}{\pi^{(m-1)}}.$$
Next, we randomly choose $y^{(m)}\sim \operatorname{Uniform}(0,1)$ and  set
$$c_{j}^{(m)} :=\begin{cases}
\ds \widetilde{c}_{j}^{(m)}\quad&\mbox{if }\alpha^{(m)}\ge y^{(m)},\\[1mm]
\ds c_{j}^{(m-1)}\quad&\mbox{if }\alpha^{(m)}<y^{(m)},\quad j=1,\dots,\widetilde{J}.
\end{cases} $$

\item[(2.c)] (Stopping criterion) In any case, we stop the iteration at $m=nM$ for some integer $M\ge1000$. If the Jaccard distance of $\Om^{(np)}$ and $\Om^{(nq)}$ for $M-1000< p,q\le M$ is small enough, we stop the iteration at $m=nM$. Otherwise, we return to Step (2.a) and repeat the process until $m=5000n$ (so that $M=5000$).  
\end{itemize}

\section{Numerical simulation}\label{sec:numerical}
In this section, we present numerical results on the MCMC sampling scheme in the monostatic measurement configuration.
The synthetic data is obtain from FEKO, which is commercial software for computational electromagnetics. This differs from the forward solver that we use in the iteration to compute $u^\infty[\widetilde{\Om}^{(m)}]$ in \eqnref{def:prob:foracceptance}.
We calibrate the measurement data by multiplying a positive constant so that the two forward solvers generate identical data. 
The number of directions is set to be $J=36$ and the frequency of the incident wave is assumed to be $f=1$GHz.
In this case, the wavenumber $k$ in the problem \eqnref{HelmholtzEquation} is $k=2\pi f\sqrt{\epsilon_0\mu_0}\approx20.9585{\rm m}^{-1}$ where we have the permittivity $\epsilon_0\approx8.8542\times 10^{-12}$F/m and permeability $\mu_0\approx1.2566\times 10^{-6}$H/m of vacuum.
Here, the wavelength is $\lambda = 2\pi/k \approx 0.2998$m.

We consider the example target domains $\Om_j$ bounded by
\beq\label{eq:omegatilde:large}
\begin{aligned}
&\p\Om_1:=\{0.01 + 0.03e^{{\rm i}\theta}\,:\,\theta\in[0,2\pi)\},\\
&\p\Om_2:=\{0.01+0.024{\rm i}\cos\theta + 0.036\sin\theta\,:\,\theta\in[0,2\pi)\},\\
&\p\Om_3:=\left\{0.7+{\rm i} +0.5(\cos\theta - 0.2\sin^2\theta + 0.9{\rm i}\sin\theta)\,:\,\theta\in[0,2\pi)\right\}.		
\end{aligned}
\eeq
The domain $\Om_3$ is the extended target with diameter greater than the wavelength $\lambda \approx 0.2998$m.
We note that the perturbed directions, $\widetilde{\Psi}_j$, were linearly independent (so that $\widetilde{J}=2J=72$) in all of the examples in this section.

\subsection{Performance for different shape parameters}
We compare the performance of the MCMC sampling for different shape parameters.
For the target domains $\Om_1$ and $\Om_2$, we set the initial shape $\Omega_0$ to be the disk $B(\mathbf{0},0.01)$ illustrated in Figure \ref{fig:initial:compbasis}.

\begin{center}
\begin{minipage}[c]{\linewidth}
\centering
\includegraphics[width=0.8\linewidth]{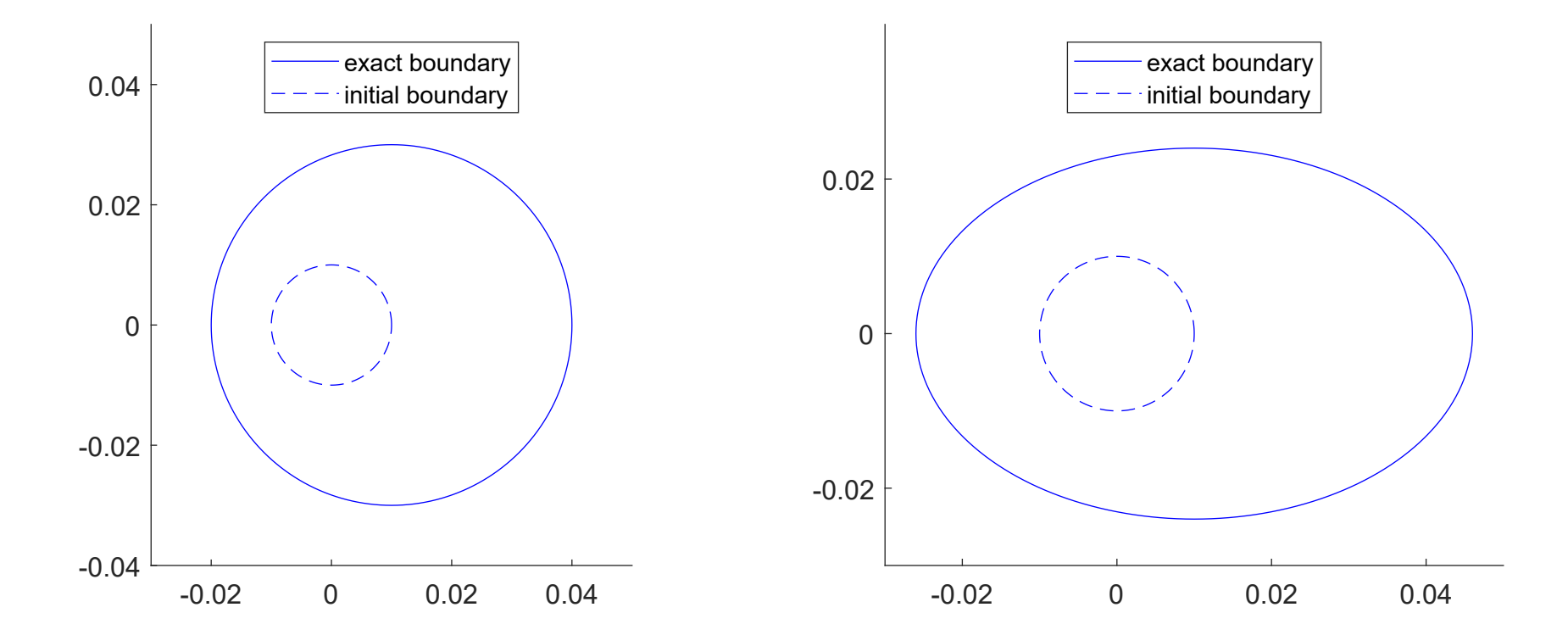}
\end{minipage}
\captionof{figure}{We illustrate the initial shapes for the MCMC sampling described in Figure \ref{fig:compare:disknellipse}.}\label{fig:initial:compbasis}
\end{center}

\begin{center}

\begin{minipage}[c]{\linewidth}
\includegraphics[width=\linewidth]{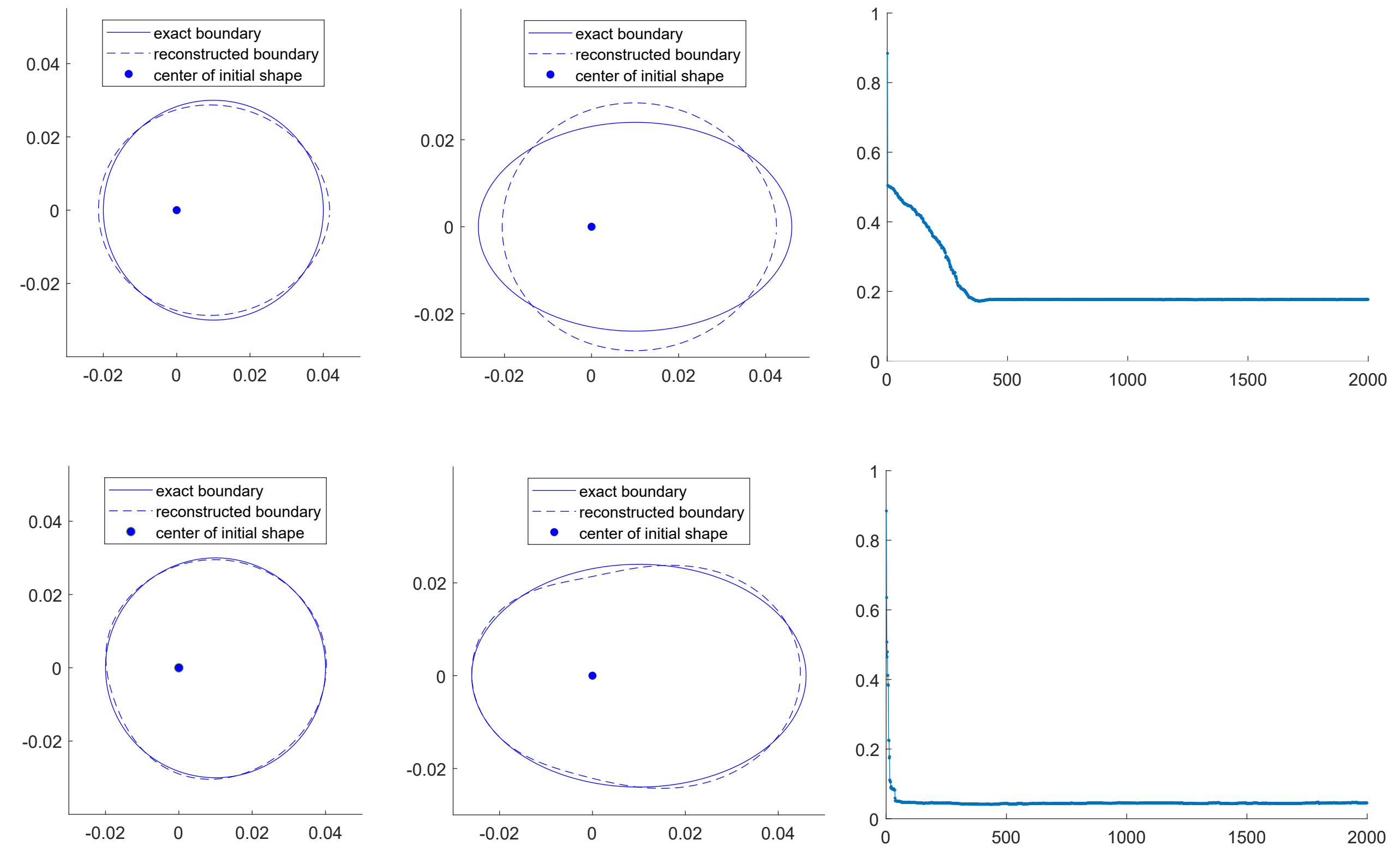}
\captionof{figure}{Result for 73 Fourier basis functions (up) and 72 new basis functions (down). On the left and middle, we illustrate the target $\p\Om_j$ (filled) for $j=1$ (left) and $j=2$ (middle) and the reconstruction (dashed) by taking mean of the last $1000$ accepted coefficients. On the right, we plot the graph of $d_J(\Om^{(nm)},\Om_2)$ against the iteration number $m$.}\label{fig:compare:disknellipse}
\end{minipage}

\end{center}

In Figure \ref{fig:compare:disknellipse}, we compare the proposed method with a modified method with an alternative choice of basis, where $\{\widetilde{\Psi_{j}}(\mathbf{y})\}_{j=1}^{\widetilde{J}}$ is replaced by the Fourier basis
$$\{(2\pi)^{-1/2}\}\cup\{\pi^{-1/2} j^{-2}\cos(j\theta)\}_{j=1}^{J}\cup\{\pi^{-1/2} j^{-2}\sin(j\theta)\}_{j=1}^{J},\quad \theta=\arg\left(\mathbf{y}-\mathbf{c}_0\right).$$
The alternative parametrization is equivalent to using the expression $q_b$ defined as equation (5.2) in \cite{Li:2020:ESB}.
We use the single component sampling scheme, which amounts to setting $L=1$ (so that $n=72$) in Step (1.a) of the previous section.
We set $\beta =2\times10^{-4}$, $\sigma = 10^{-4}$ and $\tau=0$.
As a measure of accuracy, we consider the Jaccard distance
$$d_J(\Om^{(nm)},\Om)
=1-\frac{| \Om^{(nm)}\cap\Om |}{|\Om^{(nm)}\cup\Om|}$$
for the accepted shape $\Om^{(nm)}$ at each step of the iteration.

From the left two columns of Figure \ref{fig:compare:disknellipse}, we observe that the MCMC sampling methods with the usual Fourier basis and our new basis are both successful in finding centers, but our method gives far better result in matching the boundaries of the target ellipse.
From the rightmost column of Figure \ref{fig:compare:disknellipse}, we observe that the sampling method using our basis stabilizes faster.
This can be interpreted as the effect of choosing only the directions of perturbation, $\{\widetilde{\Psi_{j}}\}_j$, that makes difference in the data $\{u^\infty(\mathbf{x}_j,-\mathbf{x}_j)\}_j$ under consideration.


\subsection{Shape reconstruction of extended targets}
\label{sec:result:large}

In this subsection, we reconstruct the extended target ${\Om}_3$.
We add the white Gaussian noise of SNRs $\infty$, $20$dB and $5$dB to the complete $J\times J$ MSR matrix and reconstruct the target domains using only the diagonal entries and the proposed method. 
Unlike in the previous subsection, we optimize the center and radius of the initial disk $B(\mathbf{c}_0,r_0)$ for MCMC iteration using the given monostatic data.
We use monostatic sampling method (MSM) in the previous studies (for example, \cite{Kang:2022:MSM}) to find the center $\mathbf{c}_0$; see Figure \ref{fig:indicator:largekite}.
\begin{center}
	\begin{minipage}[c]{\linewidth}
		\centering
		\includegraphics[width=0.9\linewidth]{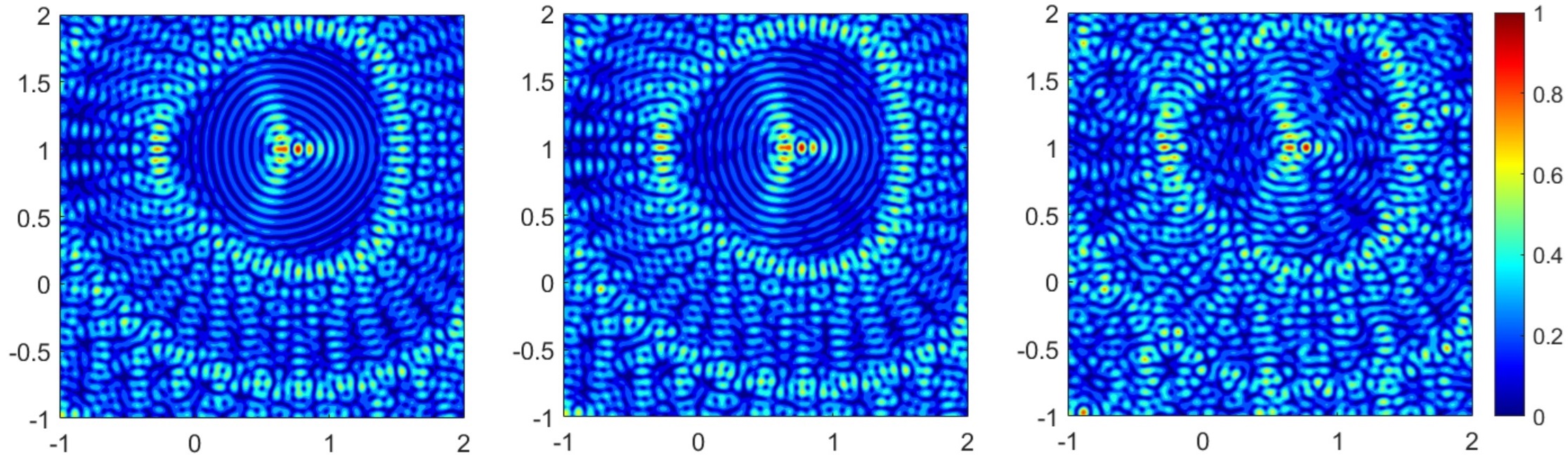}
		\captionof{figure}{Graph of the index function in MSM \cite{Kang:2022:MSM} for the target domain ${\Om}_3$ using the monostatic data with noise of SNR values $\infty$, $20$dB and $5$dB from left to right.}\label{fig:indicator:largekite}
	\end{minipage}
\end{center}
In order to find the radius $r_0$, we use the root-finding method in Appendix \ref{disk case}.
We use the value $r(\Om_3)$ in \eqnref{eq:rad:optimal} for the radius of the initial disk.
The parameters in Subsection \ref{subsec:Gibbs} are set to be $L=12$ (so that $n=6$), $\sigma = 10^{-4}$ and $\tau=10^2$.

We use a general strategy to choose $r_0$ among a few local maximizers with biggest values of the index function (see Figure \ref{fig:indicator:largekite}) that makes the MCMC iteration to stabilize at the biggest value of $\pi^{(m)}$.
We use the local maximizers instead of the global maximizers for which the index function in Figure \ref{fig:indicator:largekite} is the second largest among the local maximums. 
In Figure \ref{large:kite:corrected}, we plot the following frequency distribution as filled contours:
$$f_{N_1,N_2}(\mathbf{x}) = \frac{1}{N_2-N_1} \#\left\{m\in\mathbb{N}\,:\,\mathbf{x}\in\Om^{(nm)},\, 
N_1< m \le N_2\right\},\quad\mathbf{x}\in\mathbb{R}^2.$$

\begin{center}
\begin{minipage}[c]{\linewidth}
\centering
\includegraphics[width=0.9\linewidth]{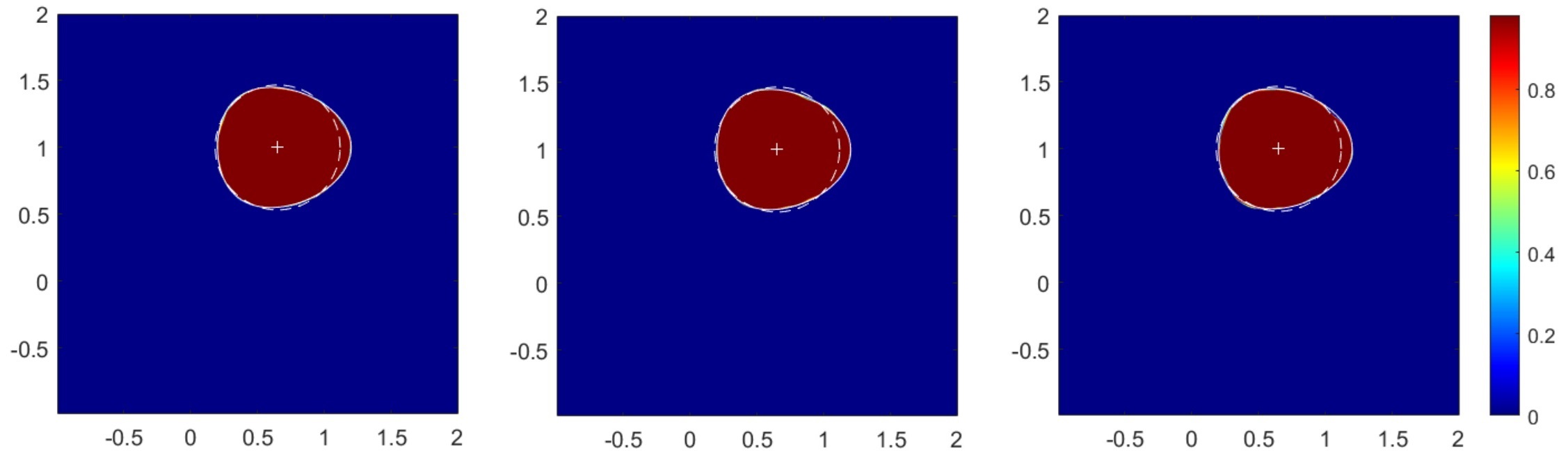}
\captionof{figure}{Shape reconstruction results for ${\Om}_3$ with optimized initial disk. From the right to the left, we plot the filled contours of $f_{N_1,N_2}$ with $N_1=2500$ and $N_2=5000$ for the monostatic measurement data with SNRs $\infty$, $20$dB and $5$dB, respectively. The white curves are the boundaries of the target domains (filled) and the initial disks $\Om_0$ (dashed). The white $+$ markers are the centers of the initial disks for the MCMC sampling.}\label{large:kite:corrected}
\end{minipage}
\end{center}

\begin{center}
	\begin{minipage}{\linewidth}
		\centering
		\includegraphics[width=0.7\linewidth]{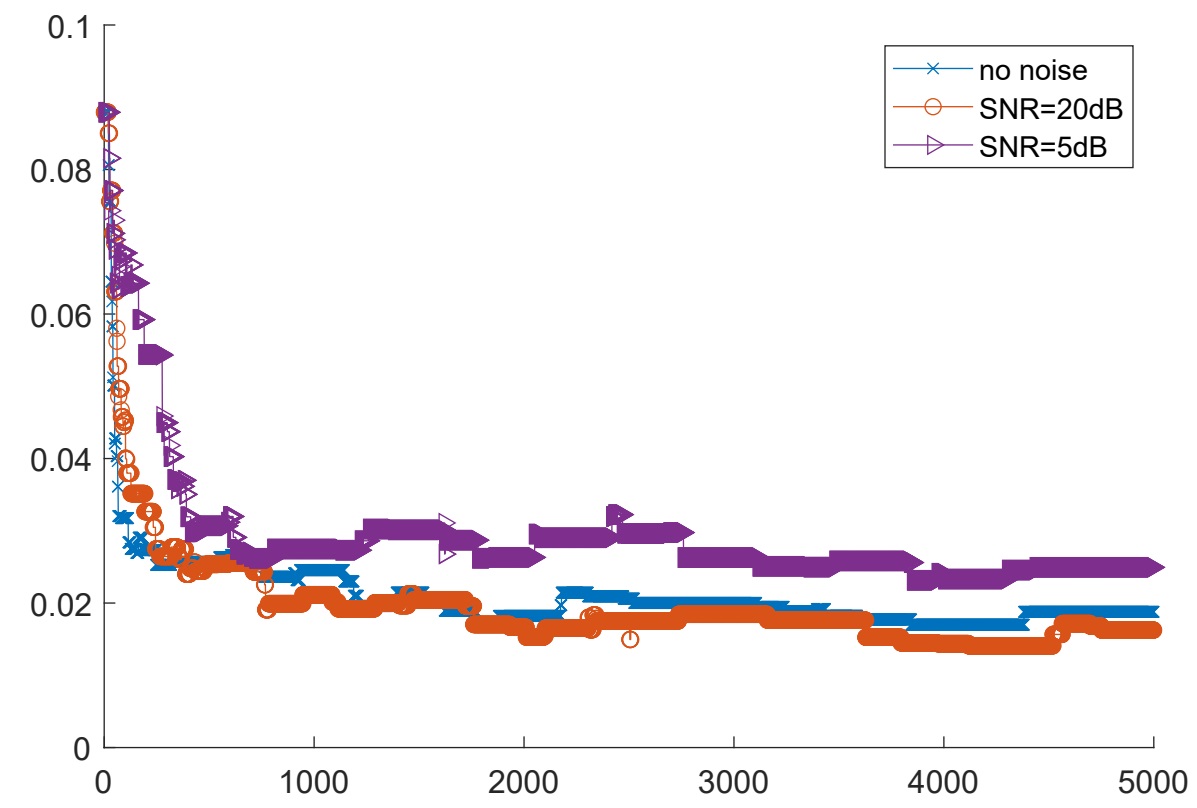}
		\captionof{figure}{Shape reconstruction results for ${\Om}_3$ with optimized initial disk. We plot the graph of $d_J(\Om^{(nm)},\Om_3)$ against the iteration number $m$.}\label{large:kite:corrected:J}
	\end{minipage}
\end{center}

\section{Conclusion}
In this paper, we propose an MCMC sampling method to predict the new shape parameters that arises in the shape derivative analysis for the measurement data.
We facilitate the sampling method by optimizing the center and radius using the index function in MSM and explicit expression of the far field patterns for disks.
In the numerical simulation, we observed that our proposed method has high accuracy and stability when the exact shape is a small perturbation of the initial disk.
For example, the proposed method showed good performance for the extended target ${\Om}_3$ in Section \ref{sec:numerical}.

\appendix
\section{Far field asymptotic for disks}
\label{disk case}
In this section, we derive an explicit expression of $u^\infty[\Om]$ for the case $\Om=B(\mathbf{x}_0,r_0)$, where $B(\mathbf{x}_0,r_0)$ denotes the disk with center at $\mathbf{x}_0$ and radius $r_0$.
Let $(r,\theta)$ be the polar coordinate system with center at $\mathbf{x}_0$.
The general solution $v$ for the Helmholtz equation $\Delta v + k^2 v=0$ ($k>0$) is
$$v(r,\theta) = \sum_{m=-\infty}^\infty \left(c_m^{(1)} H_m^{(1)}(kr) + c_m^{(2)} H_m^{(2)}(kr)\right)e^{{\rm i}m\theta},\quad c_m^{(1)},c_m^{(2)}\in\mathbb{C},$$
where $H_m^{(j)}$ is the Hankel function of $j$th kind with order $m$.
We consider the case $v=u^{\rm scat}$.
From the Dirichlet boundary condition $u^{\rm scat}=-u^{\rm inc}$ on $\p\Om$, we use Jacobi--Anger expansion for $u^{\rm inc}$ on the right-hand side, which gives
\beq\label{eq:disk:polarexp:1}
c_m^{(1)} H_m^{(1)}(kr_0)+ c_m^{(2)} H_m^{(2)}(kr_0)=-e^{{\rm i}k\mathbf{d}\cdot\mathbf{x}_0}e^{-{\rm i}m\operatorname{arg}(\mathbf{d})}{\rm i}^mJ_m(kr_0)\quad\mbox{for all }m\in\mathbb{Z},
\eeq
where $J_m$ is the Bessel function of order $m$.
Moreover, because $u^{\rm scat}$ satifies the Sommerfeld radiation condition, we can show $c_m^{(2)}=0$ for all $m\in\mathbb{Z}$ as what follows:
For all $m\in\mathbb{Z}$, we have $H_m^{(2)}(kr)=\overline{H_m^{(1)}(kr)}$ and the relations
\begin{align}
\label{eq:H1:asymp}
H_m^{(1)}(z) &\sim \sqrt{\frac{2}{\pi z}} e^{{\rm i}\left(z-\frac{m\pi}{2}-\frac{\pi}{4}\right)}\quad\mbox{as }z\to\infty,\\
\label{eq:H1:prime}
\frac{\p}{\p r}H_m^{(1)}(kr) &= \frac{k}{2}\left(H_{m-1}^{(1)}(kr)-H_{m+1}^{(1)}(kr)\right).
\end{align}
The Sommerfeld radiation condition gives, for each $m\in\mathbb{Z}$,
$$\lim_{r\to\infty}{\sqrt{r}}\left[c_m^{(1)}\left(\frac{\p H_m^{(1)}(kr)}{\p r}-{\rm i}k H_m^{(1)}(kr)\right)+c_m^{(2)}\left(\frac{\p H_m^{(2)}(kr)}{\p r}-{\rm i}k H_m^{(2)}(kr)\right)\right]=0.$$
Using \eqnref{eq:H1:asymp} and \eqnref{eq:H1:prime}, we deduce that
\begin{align*}
&\lim_{r\to\infty}{\sqrt{r}}\left(\frac{\p H_m^{(1)}(kr)}{\p r}-{\rm i}k H_m^{(1)}(kr)\right)=0,\\
&{\sqrt{r}}\left(\frac{\p H_m^{(2)}(kr)}{\p r}-{\rm i}k H_m^{(2)}(kr)\right)\sim -{\rm i}\sqrt{\frac{8k}{\pi}}\cos\left(kr-\frac{m\pi}{2}-\frac{\pi}{4}\right)\quad\mbox{as }r\to\infty.
\end{align*}
We conclude that $c_m^{(2)}=0$ for all $m\in\mathbb{Z}$.
Combining with \eqnref{eq:disk:polarexp:1}, we arrive at the following lemma:
\begin{lemma}\label{lemma:polarexp}
Let $\Om=B(\mathbf{x}_0,r_0)$ for some $\mathbf{x}_0\in\mathbb{R}^2$ and $r_0>0$. In terms of the polar coordinates $(r,\theta)$ with center at $\mathbf{x}_0$, we have, for $r\ge r_0$,
$$u^{\rm scat} = \sum_{m=-\infty}^\infty c_m^{(1)}H_m^{(1)}(kr)e^{{\rm i}m\theta},\quad\mbox{where}\quad c_m^{(1)}=-{\rm i}^me^{{\rm i}k\mathbf{d}\cdot\mathbf{x}_0}e^{-{\rm i}m\operatorname{arg}(\mathbf{d})}\frac{J_m(kr_0)}{H_m^{(1)}(kr_0)}.$$
\end{lemma}

Applying \eqnref{eq:H1:asymp} to Lemma \ref{lemma:polarexp}, we obtain
$$u^{\rm scat} = \sum_{m=-\infty}^\infty c_m^{(1)}H_m^{(1)}(kr)e^{{\rm i}m\theta}\sim \sqrt{\frac{2}{\pi kr}} \sum_{m=-\infty}^\infty c_m^{(1)} e^{{\rm i}\left(kr-\frac{m\pi}{2}-\frac{\pi}{4}\right)} e^{{\rm i}m\theta}\quad\mbox{as }r\to\infty.$$
We arrive at the following theorem.
\begin{thm}
Let $\Om=B(\mathbf{x}_0,r_0)$ for some $\mathbf{x}_0\in\mathbb{R}^2$ and $r_0>0$. In terms of the polar coordinates $(r,\theta)$ with center at $\mathbf{x}_0$, we have
$$u^{\infty}(\hat{\mathbf{x}},\mathbf{d}) = \sqrt{\frac{2}{\pi k}}\sum_{m=-\infty}^\infty e^{{\rm i}\frac{\pi}{4}} e^{{\rm i}k\mathbf{d}\cdot\mathbf{x}_0} \hat{\mathbf{x}}^m \mathbf{d}^{-m}\frac{J_m(kr_0)}{H_m^{(1)}(kr_0)},\quad \hat{\mathbf{x}},\mathbf{d}\in\{z\in\mathbb{C}\,:\,|z|=1\}.$$
\end{thm}
Using this theorem, we will construct an algorithm to choose $r_0$ in order to use $B(\mathbf{x}_0,r_0)$ as the initial shape for our iterative scheme to reconstruct the shape of the sound-soft object.

\smallskip
\smallskip
\noindent{\bf Disk shaped target.}\\
For the monostatic measurement with $\hat{\mathbf{x}}(\theta) = e^{{\rm i}\theta}$ and $\mathbf{d}(\theta)=-e^{{\rm i}\theta}$, we have
\begin{align*}
u^{\infty}(\hat{\mathbf{x}}(\theta),\mathbf{d}(\theta)) & = \frac{{\rm i}-1}{\sqrt{\pi k}}e^{-{\rm i}k(\cos\theta,\sin\theta)\cdot\mathbf{x}_0}\sum_{m=-\infty}^\infty (-1)^m\frac{J_m(kr_0)}{H_m^{(1)}(kr_0)}.
\end{align*}
We set, for simplicity,
$$f(r_0):=\sum_{m=-\infty}^\infty (-1)^m\frac{J_m(kr_0)}{H_m^{(1)}(kr_0)}=\left(\frac{\sqrt{\pi k}}{{\rm i}-1}e^{{\rm i}k(\cos\theta,\sin\theta)\cdot\mathbf{x}_0}\right)u^{\infty}(\hat{\mathbf{x}}(\theta),\mathbf{d}(\theta)).$$
In order to approximate $|f|$, we plot the function
$$f_M(r):=\left|\sum_{m=-M}^M (-1)^m\frac{J_m(kr)}{H_m^{(1)}(kr)}\right|$$
in Figure \ref{fig:fagainstr}.
The function $|f|$ seems to be monotone.
Assuming that $\Om$ is a disk, we can compute its radius by finding the root of $|f|=C$ using the bisection method, which is a guaranteed computation scheme.

The advantage of taking absolute value of $f$ lies in both the monotonicity of $|f(r)|$ in $r$ and the fact that the information of the center is not needed in determining the radius.

\begin{center}
\begin{figure}[h!]
\centering
\includegraphics[width=0.7\textwidth]{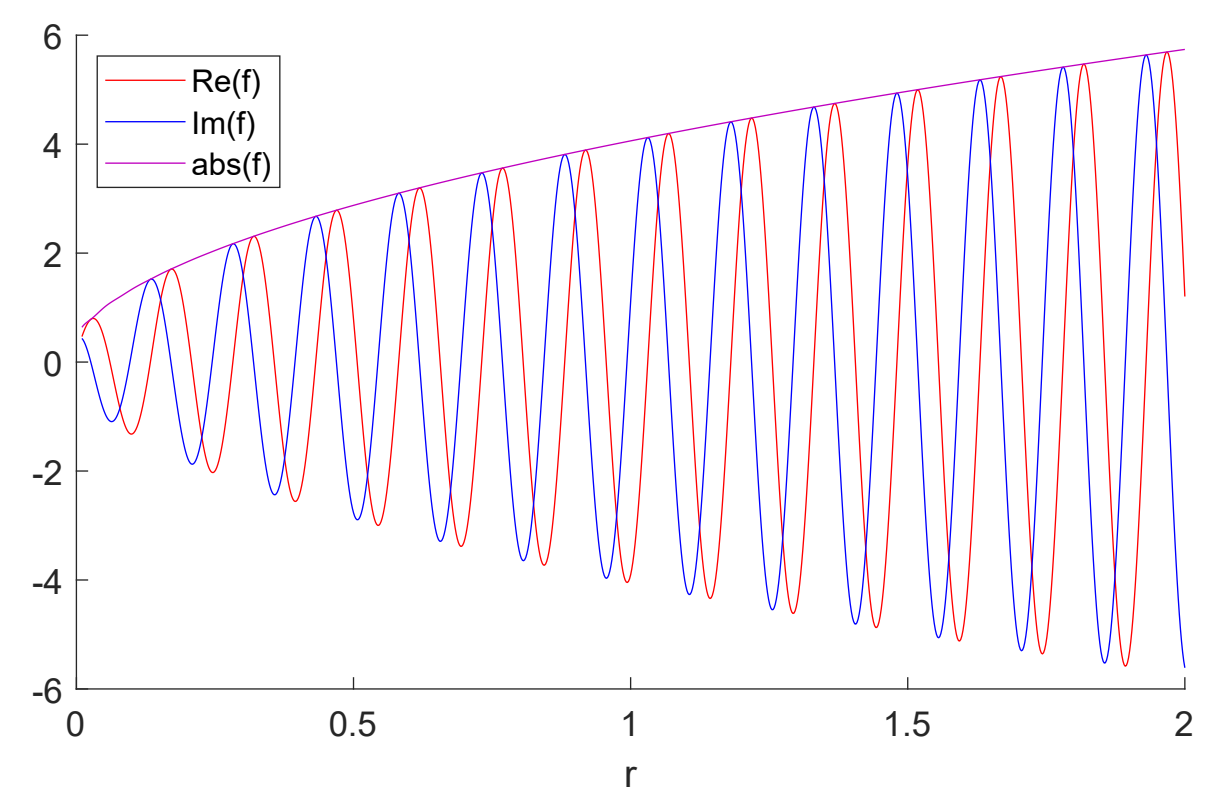}
\caption{\label{fig:fagainstr}We plot $f_M$ against $r$ with $M=200$ and $k=20.95845$.}
\end{figure}
\end{center}


\smallskip
\smallskip
\noindent{\bf General target.}\\
In this subsection, we consider $u^\infty$ for the targets $\Om$ that are not disks.
We set
$$g(\theta) := \left|\left(\frac{\sqrt{\pi k}}{{\rm i}-1}e^{{\rm i}k(\cos\theta,\sin\theta)\cdot\mathbf{x}_0}\right)u^{\infty}(\hat{\mathbf{x}}(\theta),\mathbf{d}(\theta))\right| = \frac{\sqrt{\pi k}}{\sqrt{2}}|u^{\infty}(\hat{\mathbf{x}}(\theta),\mathbf{d}(\theta))|.$$
We solve the equation $|f(r)| = \bar{g}$ to find $r$, where we set
$$\bar{g}:=\frac{1}{N}\sum_{j=0}^{N-1}g(\theta_j)\quad\mbox{or}\quad\bar{g}:=\left(\frac{1}{N}\sum_{j=0}^{N-1}g(\theta_j)^2\right)^{1/2},\quad \theta_j=\frac{2\pi j}{N},\ j=0,1,\cdots,N-1.$$
As in the case for disks, we use the bisection method.
We note that the above scheme for general shape does not require the prior knowledge on the location.
We list the computed $r$ for each target domains ${\Om}_j$ which we defined in \eqnref{eq:omegatilde:large}:
\beq\label{eq:rad:optimal}
\begin{aligned}
& r({\Om}_1)=0.0300,\quad
r({\Om}_2)=0.0299,\quad
r({\Om}_3)=0.467.
\end{aligned}
\eeq
The result is exact for ${\Om}_1$ and comparable with the sizes of the other domains.

\bibliographystyle{plain}
\bibliography{references}

\end{document}